\documentclass[11pt,letterpaper]{amsart}

\usepackage{amsmath,amssymb,amsthm, epsfig, amsmath}
\usepackage{xcolor}

\usepackage[usenames,dvipsnames]{pstricks}
\usepackage{epsfig} 
\usepackage{pst-grad} 
\usepackage{pst-plot}
\usepackage{verbatim}
\usepackage{mathtools}

\usepackage{xcolor}
\usepackage{tikz} 
\definecolor{azul}{RGB}{33,64,154}

\usepackage{hyperref}
\hypersetup{
	colorlinks,
	urlcolor=azul,
	linkcolor=azul,
	citecolor=azul,
	pdftitle={Free-boundary problems}}

\def \suchthat {~\big{|} ~}

\newtheorem{theorem}{Theorem}[section] 
\newtheorem{lemma}{Lemma}[section]
\newtheorem{proposition}{Proposition}[section]  
\newtheorem{corollary}{Corollary}[section] 
\newtheorem{definition}{Definition}[section]
\newtheorem{remark}{Remark}[section]

\numberwithin{equation}{section}

\newcommand{\Lip} {\operatorname{Lip }}

\title[Diffusion models involving gradient activation]{Regularity in diffusion models with gradient activation}

\author[D.J. Ara\'ujo]{Dami\~ao J. Ara\'ujo}
\address{UFPB, Department of Mathematics, Universidade Federal da Para\'iba, 58059-900, Jo\~ao Pessoa-PB, Brazil}{}
\email{araujo@mat.ufpb.br}

\author[A.O. Sobral]{Aelson O. Sobral}
\address{UFPB, Department of Mathematics, Universidade Federal da Para\'iba, 58059-900, Jo\~ao Pessoa-PB, Brazil}{}
\email{aelson.sobral@academico.ufpb.br }

\author[E.V. Teixeira]{Eduardo V. Teixeira}
\address{Department of Mathematics, University of Central Florida, 4393 Andromeda Loop N. Orlando, FL, USA 32816}{}
\email{eduardo.teixeira@ucf.edu}

\begin{document}
	
	
	
\begin{abstract} 
We prove sharp regularity estimates for solutions of highly degenerate fully nonlinear elliptic equations. These are free boundary models in which a nonlinear diffusion process drives the system only in the region where the gradient surpasses a given threshold. Our main result concerns the existence of a universal modulus of continuity for $Du$, {\it up to the free boundary}. Gradient bounds with respect to the $L^\infty$ norm are proven to be uniform with respect to the degree of degeneracy. Several new ingredients are needed and among the tools introduced in this paper is an improvement of regularity lemma designed to measure the oscillation decay with respect to the gradient level-set distance. Applications of the methods are discussed at the end of the paper.
\end{abstract}   
	
\date{\today}
	
\maketitle
	
\tableofcontents

\section{Introduction}\label{sct intro}

In this paper, we investigate diffusion models that are triggered by a gradient threshold. These are self-regulatory systems in which a diffusive agent is prompted whenever the density difference becomes much larger than the displacement. Mathematically, this leads to the analysis of a class of \textit{highly} degenerate elliptic partial differential equations of the form
\begin{equation}\label{Eq-General}
	\mathcal{H}(Du,D^2 u) = f,
\end{equation}
 where the operator $\mathcal{H}\colon  \mathbb{R}^n \times \text{Sym}(n) \to \mathbb{R}$ collapses in a subset $\mathcal{C} \subset \mathbb{R}^n$, corresponding to the gradient argument, i.e. $\mathcal{H}(\xi,M)\equiv 0$, for all $\xi \in \mathcal{C}$, and $M \mapsto \mathcal{H}(\xi,M)$ is elliptic for $\xi \in \mathbb{R}^n \setminus \mathcal{C}$.
 
 Problems of that nature appear, for instance, in the theory of superconductivity, when examining vortices in the mean-field model, e.g. \cite{C95,BR99,ESS98} and \cite{CSS04}. Variational interpretations are related to minimization issues in random surfaces and tilings, see \cite{random, tiling} and \cite{DS} for such a connection, as well as to problems in congested traffic dynamics, see \cite{BCS10} as well as \cite{CF14, BD23, SV10}. Fully nonlinear equations of this type also appear as limiting free boundary problems, obtained when the degeneracy parameter of the equation tends to infinity---a free boundary version of the infinity Laplacian operator if you will--- see subsection \ref{subsec limiting FB} for further details.
 
 Note the region where the system is governed by a PDE depends upon the solution itself, more precisely on its gradient. That is, the correct way to interpret \eqref{Eq-General} is as an (unconstrained) free boundary problem, viz.
 \begin{equation}\label{Eq-General-FBP}
	\mathcal{H}(Du,D^2 u) = f, \quad \text{ in } \left\{ x \in \Omega \suchthat  Du(x) \in \mathbb{R}^n \setminus \mathcal{C} \right\}.
\end{equation}
We will further discuss this point of view in subsection \ref{subsct FB}.

To simplify the presentation, we focus on the case $\mathcal{C}=B_\kappa$, for $\kappa\ge 0$, leading to the free boundary problem
\begin{equation}\label{maineq}
	(|Du|-\kappa)_+^q F(D^2 u) = f \quad \mbox{in } \; \{|Du| > \kappa \}.
\end{equation}
The operator $F$ is uniformly elliptic and the parameter, $q \geq 0$, prescribes the degeneracy degree of the model along the free boundary $\partial \{|Du| > \kappa \}$. It is worth noting that the problem is still (very) degenerate even if $q=0$, due to the diffusion collapse in the (a priori unknown) region $\{|Du| \le \kappa \}$. 

It is also important to highlight that no information upon the sets $\{|Du| \le \kappa \}$ and $\{|Du| > \kappa \}$ can be a priori inferred. In particular, the free boundary, $\partial \{|Du| > \kappa\}$ can be very irregular, and thus out of the scope of known elliptic boundary regularity estimates.

The case $\kappa = q = 0$ falls into the theory launched  by \cite{CS02}, where the authors investigated fully nonlinear elliptic equations of the form 
\begin{equation}\label{Eq-Caff-S}
    F(D^2u) = g(x,u)\chi_{\{|Du| \not = 0 \}}.
\end{equation}
Solutions of \eqref{Eq-Caff-S} are understood in a very weak viscosity sense, where one disregards smooth test functions that touches with zero gradient. In \cite{CS02}, the authors manage to show that solutions of \eqref{Eq-Caff-S} satisfy (ordinary) viscosity inequalities, and thus the classical fully nonlinear regularity theory applies.  In the case $F = \Delta$ and $g(x,u) = cu$, the authors obtain the sharp $C^{1,1}$-regularity of solutions to \eqref{Eq-Caff-S}; see also \cite{CKS} for related advances on similar problems.

In parallel to the approach adopted in \cite{CS02}, in this paper we introduce the concept of $\kappa$-grad viscosity solutions of \eqref{maineq}, see Definition \ref{def-kappa-grad}. The idea is to interpret \eqref{maineq} by disregarding test functions touching $u$ at point $x_0$ with not sufficient large slope. That is, the corresponding viscosity inequalities are enforced only at points $x_0$ for which one can touch by a smooth test function $\varphi$ verifying $|D\varphi(x_0)| > \kappa$. 

Clearly, when $\kappa > 0$, the optimal (local) regularity one can hope for a solution of \eqref{maineq} is Lipschitz continuity. This is because any function whose gradient norm is less than $\kappa$ automatically satisfies the equation.  Also, one can easily construct 1D-examples of solutions of \eqref{maineq} that are merely Lipschitz continuous. On the other hand,  $\kappa$-grad viscosity solutions of \eqref{maineq} are entitled to the regularity theory developed in \cite{IS16}. In particular, solutions are locally of class $C^{0,\alpha}$, for some $0< \alpha \ll 1$,  depending on dimension, ellipticity constants, and $\kappa$. 

The first main result of this paper is the sharp Lipschitz regularity estimate for  $\kappa$-grad viscosity solutions of \eqref{maineq}, see Theorem \ref{LipTHM}. The proof relies on carefully crafting special jets, as in \cite{CIL92}, whose gradient at touching points is sufficiently large. We perform a meticulous analysis, identifying all possible dependencies along the process. In particular, we prove that the Lipschitz norm of solutions of \eqref{maineq} does not depend upon the degree of degeneracy, $q$. We mention that this remark is new (and sharp) even in the case that the PDE holds everywhere, say for the family of PDEs:
\begin{equation}\label{modelq}
    |Du|^q F(D^2 u) = f, \quad \text{ in } B_1 \subset \mathbb{R}^n,
\end{equation} 
with $c< f< c^{-1}$. Indeed, a result proven in \cite{ART15}, see also \cite{IS13} and \cite{APPT}, assures that viscosity solutions of \eqref{modelq} are locally of class $C^{1,\frac{1}{1+q}}$ (at least for $q\gg 1$) and that such a regularity is optimal. Hence, insofar as uniform-in-$q$ estimates are concerned, gradient bounds are the best one can hope for solutions $u_q$ of \eqref{modelq}.

While Lipschitz estimates are indeed optimal in regards to local regularity of solutions to \eqref{maineq}, one could inquire about $C^1$ regularity within the PDE region, viz. $\Omega_u \coloneqq \{|Du| > \kappa \}$, up the free boundary,
$$
    \Gamma_u \coloneqq \partial \{|Du| > \kappa \}.
$$
This problem is particularly challenging, as it seems hard to say anything about the structure of $\Gamma_u$, unless further information is given; see \cite{CSS04} for the case $q=\kappa = 0$ and $F = \Delta$. 

It is worth noting that, continuity of $(|D u|-\kappa)_+$ implies $\Omega_u$ must be an open set, and that the PDE $(|Du|-\kappa)_+^q F(D^2 u) = f$ holds in the traditional viscosity sense within $\Omega_u$. 

The considerations above give rise to a slightly stronger, though necessary, notion of solutions to \eqref{maineq}, see Definition \ref{def-effective}. Under such a regime, the second main result we prove in this paper yields a universal modulus of continuity of the gradient of $u$ in $\Omega_u$, up to the free boundary, $\Gamma_u$, see Theorem \ref{C1THM}. The proof combines several ingredients and it will be delivered in Section \ref{sct-C1}. The idea relies on an interplay between interior $C^{1,\alpha_d}$ regularity estimates at points that are $d$-away (with respect to the gradient level-set distance) from the free-boundary, $\Gamma_u$, and how $0<\alpha_d \ll 1$ deteriorates as $d\to 0$. This is attained by introducing a sort of DeGiorgi's improvement of oscillation technique at the gradient-level. This is particularly useful to gauge regularity for points sufficiently close to the free boundary, with respect to the gradient level-sets. For points far from the free boundary (again with respect to the gradient level-set distance), the equation is elliptic, and thus, up to rescaling, $u$ is close to a $F$-harmonic function; uniform $C^{1,\alpha}$ regularity estimates are then obtained {\it ala} Caffarelli, \cite{C89}; see also \cite{Teix20} for a didatical account of this method.

\begin{figure}[h!]
\centering
\begin{tikzpicture}[scale=1.4]

\draw (-3,3) -- (-3,0) -- (3,0) -- (3,3) -- (-3,3);

\filldraw [black] (-0.6,0.7) circle (0.0pt) node[below left, gray] {$|Du| \sim \kappa^+$};
\filldraw [black] (-0.7,2.3) circle (0.0pt) node[below left, black] {$|Du| \leq \kappa$};
\filldraw [black] (2.5,2.3) circle (0.0pt) node[below left, gray] {$|Du| > \kappa^+$};
\filldraw [black] (2.75,1.2) circle (0.0pt) node[below left, black] {$(|Du|-\kappa)^qF(D^2 u)=f$};

\filldraw [gray, opacity=0.7] (-3,3) -- (-3,0) -- (-2.4,0) to [out=80,in=0] (-2.5,0.7) to [out=45,in=180] (-1,1)  to [out=30,in=-30] (0.2,2.5) -- (-0.1,2.3) -- (0.2,3) -- (-3,3);

\filldraw [gray, opacity=0.4] (-3,3) -- (-3,0) -- (-1.0,0) -- (0,0.5) -- (-0.5,1) to [out=60,in=180] (0.5,1.5)  to [out=90,in=180] (2.5,3) -- (-3,3);

\draw [black, opacity=0.7, very thick] (-2.4,0) to [out=80,in=0] (-2.5,0.7) to [out=45,in=180] (-1,1)  to [out=30,in=-30] (0.2,2.5) -- (-0.1,2.3) -- (0.2,3) -- (3,3) -- (3,0) -- (-2.4,0);
\end{tikzpicture}
\caption{This figure is a representation the geometry of the problem. The white region, $\{ |Du| > \kappa\}$, displays the part of the domain in which a diffusion PDE drives the system. In the dark grey zone,  $\{ |Du| \le  \kappa\}$, the system is dormant. The analysis in the intermediary light grey sector, $\{\kappa < |Du| < \kappa + \mu \}$, for some $0< \mu \ll 1$, is critical for the proof of Theorem \ref{C1THM}. It is worth highlighting, however, that the topology of such a regions can be much more complicated and their corresponding boundaries highly irregulars. This is why Theorem \ref{C1THM} is a non-trivial (somewhat striking) result.}
\label{fig2}
\end{figure}
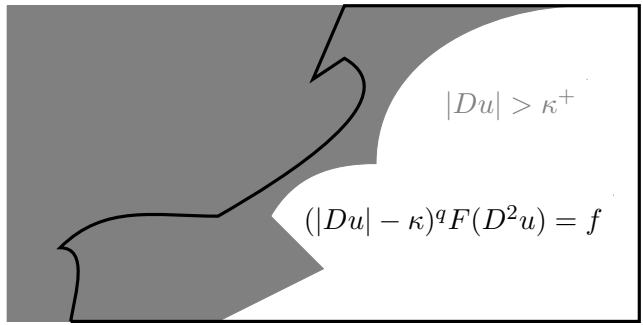

The rest of this paper is organized as follows. In Section \ref{sctPrelim} we provide the basic setup and some important concepts to be used throughout the paper. In Section \ref{sct proof lip}, we prove the uniform Lipschitz estimate, Theorem \ref{LipTHM}. In Section \ref{sct-Compactness}, we establish compactness for the scaled PDE. In Section \ref{sct-C1}, we split the analysis between the region close and far away from the free boundary to attain the universal $C^1$ regularity theorem. In the last Section \ref{sct Applic}, we discuss further applications of the methods introduced in this paper.

\section{Preliminaries}\label{sctPrelim}

In this section, we gather some classical terminologies and introduce new concepts that will be developed throughout the paper.  

Problems are modeled in the $n$-dimensional Euclidean space, $\mathbb{R}^n$. The open ball of radius $r$ centered at the point $x_0$ is denoted by $B_r(x_0)$. We shall omit the center of the ball for $x_0 = 0$. 

The space of all $n\times n$ symmetric matrices is denoted by Sym$(n)$. Given constants $0 < \lambda \leq \Lambda$, let 
$$
    \mathcal{A}_{\lambda,\Lambda} \coloneqq \left\{A \in \mbox{Sym}(n) \suchthat \lambda I_n \leq A \leq \Lambda I_n \right\}.
$$
The so-called \textit{Pucci Extremal Operators} $\mathcal{M}^+$ and $\mathcal{M}^-$, acting on Sym$(n)$, are defined as
$$
    \mathcal{M}^+(M) \coloneqq \sup_{A \in \mathcal{A}_{\lambda,\Lambda}}\text{Trace}(AM) \quad \text{ and } \quad \mathcal{M}^-(M) \coloneqq \inf_{A \in \mathcal{A}_{\lambda,\Lambda}}\text{Trace}(AM).
$$
\begin{definition}\label{UE}
Given constants $0 < \lambda \leq \Lambda$, we say that $F\colon \mbox{Sym}(n) \rightarrow \mathbb{R}$ is $(\lambda,\Lambda)$-elliptic if
$$
    \mathcal{M}^-(M-N) \leq F(M) - F(N) \leq \mathcal{M}^+(M-N), 
$$
for every $M,N \in \mbox{Sym}(n)$.
\end{definition}

\medskip

Inspired by \cite{CS02}, we propose the following definition: 

\begin{definition}\label{def-kappa-grad}($\kappa$-grad viscosity solutions)
Let $G:\mathbb{R}^n \times Sym(n) \rightarrow \mathbb{R}$ be a continuous function. Given a nonnegative $\kappa$, we say that $u$ is a $\kappa$-grad viscosity subsolution to 
\begin{equation}\label{viscosity_eq}
    G(Du, D^2u) = f
\end{equation}
if for every $x_0$ and $\varphi$ such that $(u-\varphi)$ attain a local maximum at $x_0$ with $|D\varphi(x_0)|>\kappa$ there holds
$$
    G(D\varphi(x_0), D^2\varphi(x_0)) \geq f(x_0).
$$
We say $u$ is a $\kappa$-grad viscosity supersolution for \eqref{viscosity_eq}, if for every $x_0$ and $\varphi$ such that $(u-\varphi)$ attain a local minimum at $x_0$ with $|D\varphi(x_0)|>\kappa$ there holds
$$
    G(D\varphi(x_0),D^2\varphi(x_0)) \leq f(x_0).
$$
We say $u$ is a $\kappa$-grad viscosity solution for \eqref{viscosity_eq} if $u$ is both a $\kappa$-grad subsolution and supersolution.
\end{definition}

Similarly, we will say that a continuous $v\colon {B}_1 \rightarrow \mathbb{R}$ satisfies $|Dv|(x_0) > \kappa$ (in the viscosity sense) if there exists a $C^2$ function $\varphi$ touching $v$ from above (or below) at $x_0$ satisfying $|D\varphi(x_0)|>\kappa$.

\begin{definition}
Given a continuous function $v\colon \overline{B}_1 \rightarrow \mathbb{R}$ we define 
$$
    \Omega_v = \left\{x \in B_1 \suchthat |Dv|>\kappa \right\}.
$$
The interior boundary of this set, will be denoted by $\Gamma_v$, i.e. 
$$
    \Gamma_v \coloneqq \partial \Omega_v \cap B_1
$$
\end{definition}

For the PDE model we will investigate in this paper, $\Gamma_u$ will represent the free boundary of the problem, whereas $\Omega_u$ is the region in which the system is driven by a (fully nonlinear, degenerate) elliptic equation.

We note that the notion of $\kappa$-grad viscosity solutions is indeed very weak. It enlarges the set where we search for solutions by disregarding test functions whose slope at a touching point is less than or equal to $\kappa$. In particular, this definition gives very little information about the set $\Omega_u$, where the PDE is placed. If one seeks for further regularity of solutions to \eqref{maineq} within $\Omega_u$, a bit more structure is naturally required. This is the contents of the next definition:

\begin{definition}\label{def-effective} We say $u$ is an effective viscosity solution of \eqref{maineq}, if the set $\Omega_u \coloneqq \left\{ x\in B_1 \suchthat |Du| > \kappa \right\}$ is open and $u$ satisfies 
$$
    (|Du|-\kappa)_+^q F(D^2 u) = f \quad \text{ in } \Omega_u,
$$
in the classical viscosity sense.
\end{definition}

As a byproduct of the results to be proven in this paper, $\kappa$-grad viscosity solutions of \eqref{maineq} can be easily obtained through a limiting process. More precisely, let $u_j$ be a bounded family of viscosity solutions to
\begin{equation}\label{regmaineq}
	\left((|Du_j|-\kappa)_+^q +1/j \right)F(D^2 u_j) = f, \quad \text{ in } \Omega.
\end{equation}
The regularity estimates established in this paper are uniform with respect to the approximation parameter $j$. Hence, up to a subsequence, one can pass the limit as $j\rightarrow \infty$ in \eqref{regmaineq}. It is standard to verify that the limit function will enjoy the same regularity estimate of $u_j$, i.e. Lipschitz continuous, and it solves \eqref{maineq} in the $\kappa$-grad viscosity sense. The $C^1$ regularity of $u_j$, up to the free boundary, viz. the corresponding Theorem \ref{C1THM}, is too uniform with respect to the parameter $j$.

Next, we comment on the scaling properties of the model, which shall be used throughout the entire evolution of the paper. 

\begin{remark}\label{norm}
Let $u$ be a $\kappa$-grad (resp. effective) viscosity solution of \eqref{maineq} in $B_1$. Assume $\kappa > 0$ and define the constants:
$$
    A=\frac{1}{\max(1,\|u\|_\infty)} \quad \mbox{and} \quad B=\frac{\tau \cdot \max(1,\|u\|_\infty)}{\kappa},
$$
for an arbitrary $\tau>0$. In the sequel, define 
$$
    w(x) \coloneqq A u(Bx).
$$
One easily verifies that $w$ is a $\tau$-grad (resp. effective) viscosity solution of the re-scaled model:
$$
    (|Dw|-\tau)_+^q\overline{F}(D^2w) = \overline{f},
$$
in the ball $B_{1/B}$, where 
$$
    \overline{F}(M) = (AB^2) F( (AB^2)^{-1} M)
$$
and
$$
    \overline{f}(x) = A^{q+1}B^{q+2}f(Bx).
$$
Indeed, if $\varphi \in C^2$ touches $w$ from above(or below) at a point $x$ with $|D\varphi(x)|>\tau$, then the function $\overline{\varphi}(x) = A^{-1}\varphi(B^{-1}x)$ touches $u$ from above(or below) at $Bx$ with $|D\overline{\varphi}(Bx)|>\kappa$.
\end{remark}

In view of the previous remark, all results in this paper will be proven, with no loss of generality, for normalized solution, $-1\le u \le 1$. In Section \ref{sct-C1}, we will use this remark to restrict the analysis to the case when $0<\kappa$ is a universally small constant, to be chosen {\it a posteriori}. As pointed out in Remark \ref{norm}, this is not 
restrictive. 

\section{Uniform Lipschitz estimates}\label{sct proof lip}

This section discusses the proof of sharp Lipschitz regularity of $\kappa$-grad viscosity solution of \eqref{maineq}. The main result is the following:

\begin{theorem}\label{LipTHM}
Let $u$ be a $\kappa$-grad viscosity solution of \eqref{maineq} in $B_1$. Then $u$ is Lipschitz continuous in $B_{1/2}$, with universal bounds. More precisely, there exists a constant $C$ depending only on $n$, $\lambda$, $\Lambda$, $\kappa$, $\|f\|_\infty$ and $\|u\|_\infty$, but not on $q$, such that
$$
\sup_{x,y \in B_{1/2}}\frac{|u(x) - u(y)|}{|x-y|} \leq C.
$$
\end{theorem}

As commented, Theorem \ref{LipTHM} is optimal, even in the case $q=0$. It is also important to highlight that the Lipschitz bound does not depend on the degeneracy parameter, $q$. This is interesting (and new) even in the case when the PDE holds everywhere in the domain. We will further discuss this in Section \ref{sct Applic}. 

The first key Lemma in the proof of Theorem \ref{LipTHM} fosters useful bounds for barriers, to be crafted, at maximum points of the double-variable function $w(x,y) \coloneqq u(x) - u(y)$.

\begin{lemma}\label{jensen_ishii_lemma}
Let $u$ be a $\kappa$-grad viscosity solution of \eqref{maineq} and consider double-variable functions:
\begin{equation}\nonumber
    w(x,y) = u(x) - u(y) \quad \mbox{and} \quad \varphi(x,y) \coloneqq L\phi(|x-y|) + K(|x|^2 + |y|^2),
\end{equation}
for positive parameters $L,K$ and $\phi \in C^2(\mathbb{R}^+)$ a nonnegative function. Let $(\overline{x},\overline{y})$ be an interior maximum point for $w-\varphi$ such that $\overline{x} \neq \overline{y}$. Then, 
\begin{equation}\nonumber
\begin{array}{c}
    -4\phi''(|\overline{x}-\overline{y}|)L  \; \leq \\[0.25cm]
    \displaystyle 4n \frac{\Lambda}{\lambda}K + \frac{1}{\lambda}\|f\|_\infty \left[ \left(\left|D_x \varphi(\overline{x},\overline{y})\right|-\kappa \right)_+^{-q} +\left(\left|D_y \varphi(\overline{x},\overline{y})\right|-\kappa \right)_+^{-q} \right].
\end{array}
\end{equation}
\end{lemma}

\begin{proof}
Consider 
\begin{equation}\nonumber
D_x \varphi(\overline{x},\overline{y}), \;  D_y \varphi(\overline{x},\overline{y}) \in \mathbb{R}^n \backslash \overline{B}_\kappa.
\end{equation}
From Jensen-Ishii's Lemma \cite[Theorem 3.2]{CIL92}, there exist $X,Y \in \mathcal{S}(n)$, such that
\begin{equation}\label{limiting_sub_super}
\displaystyle (|D_x \varphi(\overline{x},\overline{y})| - \kappa)^q_+F(X) \geq f(\overline{x}) \quad \mbox{and} \quad \displaystyle (|D_y \varphi(\overline{x},\overline{y})| - \kappa)^q_+F(Y) \leq f(\overline{y}).
\end{equation}

In addition, 
    \begin{equation}\label{viscosity_ineq}
\left[
  \begin{array}{cc}
  X &  0  \\[0.2cm]
  0 &  -Y  
  \end{array}
\right]
\leq
\left[
  \begin{array}{cc}
  Z &  -Z  \\[0.2cm]
  -Z &  Z  
  \end{array}
\right]
+
(2K + \iota)I_{2n\times 2n},
\end{equation}
where $Z = L D^2_x\phi(|\cdot|)(\overline{x}-\overline{y})$. Estimate \eqref{viscosity_ineq} applied to vectors $(\xi,\xi)$, provides $spec(X - Y) \subset (-\infty, 4K + 2\iota]$. On the other hand, now choosing $(\hat{\eta}, -\hat{\eta})$, for $\hat{\eta} = (\overline{x} - \overline{y})/|\overline{x} - \overline{y}|$, gives
\begin{equation}\nonumber
\begin{array}{lll}
    \displaystyle (X-Y)\hat{\eta}\cdot \hat{\eta} & \leq & \displaystyle 4Z\hat{\eta} \cdot \hat{\eta} + (4K + 2\iota)\\[0.3cm]
    & = & \displaystyle 4L\phi''(|\overline{x} - \overline{y}|) + 4K + 2\iota.
\end{array}
\end{equation}
This implies that at least one eigenvalue of $(X-Y)$ should be less than 
$$
4L\phi''(|\overline{x} - \overline{y}|) + 4K + 2\iota.$$
Therefore,
\begin{equation}\nonumber
\begin{array}{lll}
    \mathcal{M}^+(X - Y)  & \leq &\displaystyle \Lambda(n-1)(4K + 2\iota) + \lambda(4L\phi''(|\overline{x} - \overline{y}|) + 4K + 2\iota) \\[0.4cm]
    & = & \displaystyle n\Lambda(4K + 2\iota) + 4\lambda L\phi''(|\overline{x} - \overline{y}|).
\end{array}    
\end{equation}
From \eqref{UE} and \eqref{limiting_sub_super}, we conclude
\begin{equation}\nonumber
    \begin{array}{c}
      \displaystyle -\|f\|_\infty\left[ \left(\left|D_x \varphi(\overline{x},\overline{y})\right|-\kappa \right)_+^{-q} +\left(\left|D_y \varphi(\overline{x},\overline{y})\right|-\kappa \right)_+^{-q} \right] \leq \mathcal{M}^+(X-Y),     
    \end{array}
\end{equation}
and the Lemma is proven.
\end{proof}

We are ready to deliver a proof of Theorem \ref{LipTHM}; extra care is required to keep track of all constants' dependence. 

\begin{proof}[Proof of Theorem \ref{LipTHM}]
The idea is to show the existence of universal positive parameters $L$ and $K$, such that
\begin{equation}\label{tacobell}
u(x) - u(y) \leq L|x-y| +K\left(|x|^2 + |y|^2\right),
\end{equation}
for each $(x,y) \in B_{1/2}\times B_{1/2}$. 

Let us denote
\begin{equation}\label{barrier_func}
    \phi(t) = \frac{3t - 2t^{3/2}}{3}
\end{equation}
for $t \in [0,1]$. We further define
\begin{equation}\nonumber
    M \coloneqq \sup_{x,y \in \overline{B}_{1/2}}\left(u(x) - u(y) - L\phi(|x-y|) - K\left(|x|^2 - |y|^2\right)\right).
\end{equation}
Note that showing $M \leq 0$ yields \eqref{tacobell}. The strategy is then to assume that $M > 0$ and verify that this implies a constraint to the size of $L$ and $K$. 

Let $(\overline{x},\overline{y})$ be the point in which $M$ is attained. Since $\phi(0)=0$, we easily see that $\overline{x} \not = \overline{y}$. Additionally,
$$
    L \phi(|\overline{x}-\overline{y}|) + K \left(|\overline{x}|^2 + |\overline{y}|^2\right) < u(\overline{x}) - u(\overline{y}) \leq 2.
$$
This implies that, choosing $K$ universally large, there holds $|\overline{x} - \overline{y}| \leq 1/4$. Also, 
$$
    \frac{1}{2} \leq \phi'(|\overline{x} - \overline{y}|) \leq 1
$$
and thus, for $L \geq 4K$, we have 
\begin{equation}\label{hilton}
    \frac{L}{4} \leq \frac{L}{2} - K \leq \min\{|D_x \varphi(\overline{x},\overline{y})|, |D_y \varphi(\overline{x},\overline{y})|\},
\end{equation}
where, hereafter, 
$$
    \varphi(x,y)\coloneqq L \phi(|x-y|) + K(|x|^2 + |y|^2).
$$
From Lemma \ref{jensen_ishii_lemma} and the fact that  $\phi''(|\overline{x} - \overline{y}|) < -1$, we derive
\begin{equation}\label{pde}
    L  \leq \displaystyle n \frac{\Lambda}{\lambda}K + \frac{1}{4\lambda}\|f\|_\infty\left(\left(\left|D_x\varphi(\overline{x},\overline{y})\right|-\kappa \right)_+^{-q} +\left(\left|D_y\varphi(\overline{x},\overline{y})\right|-\kappa \right)_+^{-q} \right).
\end{equation}
Taking in account the last two estimates, we obtain
\begin{equation}\nonumber
    L - n\frac{\Lambda}{\lambda} K \leq \frac{1}{\lambda}\, \|f\|_\infty \left(\frac{L}{4}-\kappa \right)^{-q}_+
\end{equation}
For $L > 4(1+\kappa)$, we conclude
\begin{equation}\label{contr_ineq}
    L < K n \frac{\Lambda}{\lambda} + \frac{1}{\lambda}\, \|f\|_\infty.
\end{equation}
Thus, if one selects
$$
    L > \max \left\{4(1+\kappa), 4K, K n\frac{\Lambda}{\lambda} + \frac{1}{\lambda}\|f\|_\infty \right\},
$$
we conclude $M$ cannot be a positive quantity and the proof of Theorem \ref{LipTHM} is complete.
\end{proof}

\section{Compactness for scaled PDEs}\label{sct-Compactness}

In this section, we establish equicontinuity estimates for normalized solutions of 
\begin{equation}\label{p-regmaineq}
    (|\xi + \vartheta Du|-\kappa)_+^qF(D^2u) = f.
\end{equation}
The main goal is to obtain estimates that are independent of $\xi \in \mathbb{R}^n$ and of $\vartheta>0$. 

We note that the aforementioned equation is understood in the $\kappa$-grad viscosity sense for the auxiliary function $v(x) = \xi \cdot x + \vartheta u(x)$ with respect to the PDE
\begin{equation}\label{scaled_eq}
    (|Dv|-\kappa)_+^qF_{\vartheta}(D^2 v) = f_\vartheta,
\end{equation}
where $F_\vartheta(M) = \vartheta F(\vartheta^{-1}M)$ and $f_\vartheta = \vartheta f$. That is, saying $u$ verifies \eqref{p-regmaineq} means that $v$ is a $\kappa$-grad viscosity solution of \eqref{scaled_eq}. With that understood, we pass to discuss the first technical lemma needed to obtain uniform compactness for such PDEs.
 
\begin{lemma}\label{bustelolemma}
Assume $u$ is normalized and satisfies \eqref{p-regmaineq} with $\|f\|_\infty \leq 1$. Given $\mu \in (0,1)$, there exists a constant $C$ depending only on $n$, $\lambda$ and $\Lambda$, such that if 
\begin{equation}\label{mu_cond}
    |\xi| \geq \kappa + 2\mu \quad \mbox{and} \quad \vartheta \leq  \frac{\mu^{q+1}}{2C},    
\end{equation}
then
\begin{equation}\nonumber
    \sup_{x,y \in B_{1/2}}\frac{|u(x)-u(y)|}{|x-y|} \leq C\mu^{-q}.
\end{equation}
\end{lemma}

\begin{proof} 
The proof follows the lines of reasoning employed in Section \ref{sct proof lip}. We will only comment on the necessary amendments. 

Consider $\phi$ as defined in \eqref{barrier_func} and
$$
    M \coloneqq \sup_{x,y \in \overline{B}_{1/2}}\left(u(x) - u(y) - L\phi(|x-y|) - K\left(|x|^2 - |y|^2\right)\right).
$$
Let $(\overline{x},\overline{y})$ be the pair where $M$ is attained and assume $M>0$. First, we localize the points where $M$ is attained by choosing $K$ large enough. 

The auxiliary function $v(x) = \xi \cdot x + \vartheta u(x)$ solves \eqref{scaled_eq} in the $\kappa$-grad viscosity sense, thus  we can apply Lemma \ref{jensen_ishii_lemma} with $\varphi(x,y) \coloneqq L\phi(|x-y|) + K(|x|^2 + |y|^2)$ as to reach
\begin{equation}\label{ineq_in_lemma}
    L \leq n\frac{\Lambda}{\lambda}K + \frac{1}{4\lambda}\left( \left(\left|\vartheta D_x \varphi(\overline{x},\overline{y})+\xi\right|-\kappa \right)_+^{-q} +\left(\left|\vartheta D_y \varphi(\overline{x},\overline{y})-\xi\right|-\kappa \right)_+^{-q} \right).
\end{equation}
From \eqref{mu_cond} and the estimate
$$
    \max\{|D_x \varphi(\overline{x},\overline{y})|, |D_y \varphi(\overline{x},\overline{y})| \} \leq 2L,
$$
there holds
\begin{equation}\nonumber
    \min\{|\xi + \vartheta D_x\varphi(\overline{x},\overline{y})|, |\xi - \vartheta D_y\varphi(\overline{x},\overline{y})|\} \geq  \kappa+\mu.
\end{equation}
Therefore, from estimate \eqref{ineq_in_lemma}, we can further estimate
\begin{equation}\nonumber
    L < n\frac{\Lambda}{\lambda}K + \frac{1}{2\lambda}\mu^{-q} \leq \overline{C} \mu^{-q},    
\end{equation}
for $\overline{C}=C(n,\lambda,\Lambda)$. 
The conclusion is that if $L \geq \overline{C} \mu^{-q}$, then $M\le 0$, which is equivalent to the thesis of the Lemma.
\end{proof}

\section{$C^1$ regularity up to the free boundary}\label{sct-C1}

In this Section, we establish gradient continuity for effective viscosity solutions of \eqref{maineq}, viz Definition \ref{def-effective}. Some of the technical lemmas to be presented here, though, are still valid for the weaker notion of solutions, according to Definition \ref{def-kappa-grad}. We will state such results in their more general form for future references. 

We further comment that in this section we will deal with the solutions of \eqref{maineq} for a universally small $\kappa>0$, to be chosen later in the proof. According to Remark \ref{norm}, this is not restrictive.  The main result of this section reads as follows:

\begin{theorem}\label{C1THM}
Let $u$ be an effective viscosity solution of \eqref{maineq} in $B_1$. Then, there exists a modulus of continuity $\sigma$, depending on $\kappa$, $q$, $n$, $\lambda$, $\Lambda$, $\|f\|_{Lip}$ and $\|u\|_\infty$, such that 
$$
    (|D u|-\kappa)_+ \in C^{0, \sigma}(B_{1/2}).
$$
\end{theorem}

We comment that the main new information given by Theorem \ref{C1THM} is that $u$ is uniformly in $C_{\text{loc}}^1$ in $\Omega_u$, up to the free boundary $\Gamma_u$; a non-trivial result, as no information can be retrieved from the local structure of $\Gamma_u$. Throughout this section, we shall obtain a slightly stronger result, from which Theorem \ref{C1THM} follows as a consequence. We state it here for future reference.

\begin{proposition}\label{unif_c1_thm}
Let $u$ be an effective viscosity solution of \eqref{maineq}. Then, given $0<\mu<1$, there exist constants $\alpha_\mu \in (0,1)$ and $C_\mu>0$, depending only upon $n$, $q$, $\lambda$, $\Lambda$, $\|f\|_{Lip}$, $\|u\|_\infty$ and $\mu$, such that 
$$
    \|(|Du| - (\kappa + \mu))_+\|_{C^{0,\alpha_\mu}(B_{1/2})} \leq C_\mu.
$$
\end{proposition}

Critical to Proposition \ref{unif_c1_thm} is the fact that, while the H\"older exponent $\alpha_\mu$ may degenerate as $\mu \to 0$, the estimate is local, i.e. holds within $B_{1/2}$, and not only in the region where the PDE drives the system.

The proof of Theorem \ref{unif_c1_thm} will be divided into two main steps: given $0< \mu <1$, we slice $\Omega_u$ as follows
$$
    \Omega_u = \left\{x \in B_1 \suchthat \kappa < |Du| < \kappa + \mu \right\} \cup \left\{x \in B_1 \suchthat |Du| > \kappa + \mu \right\}.
$$
At points $\mu$-close (in the sense of level set of $|Du|$) to the free boundary $\Gamma_u$, we employ a \textit{De Giorgi} based argument to get improvement of oscillation for functions of $Du$, which corresponds to subsection \ref{subsect small grad}. At points $\mu$-far away from the free boundary, the equation is uniformly elliptic, so one can proceed with an approximation argument.

\subsection{Improvement of oscillation near the free boundary}\label{subsect small grad} Hereafter in this section we assume the source term $f$ to be a Lipschitz continuous function. Note that if $u$ is an effective viscosity solution of \eqref{maineq}, then it is locally of class $C^{1,\alpha}$ in $\{ |Du| > \kappa \}$. 

\begin{lemma}\label{sub_equation}
Let $u$ be an effective viscosity solution of \eqref{maineq} with $f\in \Lip(\bar{B}_1)$. For a unit vector $e \in \partial B_1$, consider $w$ to be defined as
$$
    w = (\partial_e u - (\kappa+\mu))_+.
$$
Then, $w$ satisfies
$$
    \mathcal{M}^+(D^2w)+q\,\mu^{-q-1}\,\|f\|_{\infty} |Dw| \geq - \mu^{-q}\,\|f\|_{Lip},
$$
in the viscosity sense in $B_1$.
\end{lemma}
\begin{proof}
To ease notation, let $\mathcal{G}$ be defined by $\mathcal{G}(\xi) = (|\xi|-\kappa)^q_+$. Notice that for $|\xi|>\kappa$
\begin{equation}\label{cond_on_grad_dep}
    |D\mathcal{G}(\xi)| \leq q(|\xi|-\kappa)_+^{q-1}.
\end{equation}

Differentiating the equation with respect to $e \in \partial B_1$ inside the open set $\left\{x \in B_1\suchthat w>0\right\}$, we obtain
\begin{equation}\nonumber
    D\mathcal{G}(Du)\cdot D(\partial_{e}u)F(D^2u) + \mathcal{G}(Du)F_{ij}(D^2u)\partial_{ij}(\partial_{e}u) = \partial_{e}f.
\end{equation}
Taking into account that $F(D^2u) = f[\mathcal{G}(Du)]^{-1}$ and dividing the above equation by $\mathcal{G}(Du)$ we get
\begin{equation}\label{DIFF_EQ}
    D\mathcal{G}(Du)\cdot D(\partial_{e}u)[\mathcal{G}(Du)]^{-2}f + F_{ij}(D^2u)\partial_{ij}(\partial_{e}u) = \partial_{e}f\,[\mathcal{G}(Du)]^{-1}.
\end{equation}
Now, from \eqref{cond_on_grad_dep} and the fact that 
$$
    \left\{x \in B_1 \suchthat w(x) > 0 \right\} \subset \left\{x\in B_1 \suchthat |Du(x)| > \kappa + \mu \right\},
$$ 
we obtain
\begin{equation}\nonumber
\begin{array}{lll}
  D\mathcal{G}(Du)\cdot D(\partial_{e}u)[\mathcal{G}(Du)]^{-2}\,f
    & \leq  & \|f\|_\infty \,[\mathcal{G}(Du)]^{-2}\,|D\mathcal{G}(Du)| \,|D(\partial_{e}u)|\\[0.2cm]
    & \leq  & q\,\mu^{-q-1}\,\|f\|_\infty  \,|D(\partial_{e}u)|\\[0.2cm]
    & = & q\,\mu^{-q-1}\,\|f\|_\infty  \,|D w|.
\end{array}
\end{equation}
Moreover, by definition of $\mathcal{G}$, we have
\begin{equation}\nonumber
    \partial_e f\,[\mathcal{G}(Du)]^{-1} \geq -\|Df\|_\infty \, \mu^{-q}.
\end{equation}
Hence, ellipticity of $F$ yields
\begin{equation}\label{FINAL_EQUATION_FOR_W}
    q\,\mu^{-q-1}\,\|f\|_\infty  \,|D w| + \mathcal{M}^+(D^2 w) \geq -\mu^{-q} \, \|Df\|_\infty,
\end{equation}
as desired.
\end{proof}

Next, we obtain an oscillation improvement of the gradient, away from (but arbitrarily near) the free boundary $\Gamma_u$. In order to ease presentation throughout this section, we adopt the following notation for a vector $e \in \partial B_1$:
$$
    w_e \coloneqq (\partial_e u - (\kappa+\mu))_+ \quad \mbox{and} \quad w_M \coloneqq (|Du| - (\kappa+\mu))_+
$$

\begin{lemma}\label{improv_oscil_step_1}
Assume $u$ is an effective viscosity solution of \eqref{maineq}, with $f\in \text{Lip}(\bar{B}_1)$. Assume that for some $\eta>0$, there holds
\begin{equation}\label{grad_small_in_measure}
	\sup_{e \in \partial B_1}\left|\left\{x \in B_{1/8}\suchthat w_e \geq (1-\eta)\|w_M\|_{L^\infty(B_{1/4})} \right\}\right| \leq (1-\eta)\left|B_{1/8}\right|.
\end{equation}
Then, there exist parameters $\overline{c}$, depending only on $n$, $\lambda$, $\Lambda$, $q$, $\mu$, $\|f\|_\infty$, and $\theta$, depending on $n$, $\lambda$, and $\Lambda$, such that
$$
    \|w_M\|_{L^\infty(B_{1/4})} \leq \max\left\{(1-\overline{c}\,\eta^{1+\,\frac{1}{\theta}} )\|w_M\|_{L^\infty(B_{1/4})} \, , \,  \left(\overline{c}\,\eta^{\frac{1}{\theta}+1}\right)^{-1}\|f\|_{Lip}\right\}.
$$
\end{lemma}

\begin{proof}
Let us call 
$$
    \mathcal{A} \coloneqq \left\{x\in B_{1/8}\suchthat w_e \geq (1-\eta)d \right\}.
$$
Easily one notes that 
$$
    \overline{w} \coloneqq \|w_M\|_\infty - w_e \geq 0,
$$
where $\|w_M\|_\infty \coloneqq \|w_M\|_{L^\infty(B_{1/4})}$. Combining Lemma \ref{sub_equation} and the weak Harnack inequality, see for instance \cite[Theorem 4.5]{KS09}, we obtain 
$$
	\|\overline{w}\|_{L^\theta(B_{1/8})} \leq C \left( \inf_{B_{1/8}}\overline{w} + \|f\|_{Lip} \right),
$$
for some $\theta=\theta(n,\lambda,\Lambda)$, and $C=C(\mu,\|f\|_\infty,n,\lambda,\Lambda,q)$. From the last inequality and \eqref{grad_small_in_measure}, we obtain
$$
    \overline{w} + \|f\|_{Lip} \geq \displaystyle C^{-1}  \left(\int_{B_{1/8}} \overline{w}\,^\theta dx \right)^{1/\theta} \geq  C^{-1}\left(\int_{\mathcal{A}^c} \left(\|w_M\|_\infty - w_e\right)^\theta dx\right)^{1/\theta},
$$
and thus,
$$
    \overline{w} + \|f\|_{Lip} \, \geq \, C^{-1} |\mathcal{A}^c|^\frac{1}{\theta}  \eta \|w_M\|_{\infty} \, \geq \, c_1 \eta^{\frac{1}{\theta}+1}  \eta \|w_M\|_{\infty}, 
$$
for some $c_1=c_1(\mu,\|f\|_\infty,n,\lambda,\Lambda,q)$. This implies that
$$
    \|w_M\|_\infty - w_e  \geq c_1 \eta^{\frac{1}{\theta}+1} \|w_M\|_\infty - \|f\|_{Lip},
$$
which translates into
\begin{equation}\label{ineq}
	\|w_M\|_\infty - w_e \geq c_1 \eta^{\frac{1}{\theta}+1} \|w_M\|_\infty - \|f\|_{Lip} \quad \mbox{in $B_{1/8}$}.
\end{equation}

Next, we split the analysis into two cases. First, we assume
$$
	c_1 \eta^{\frac{1}{\theta}+1} \|w_M\|_\infty \geq  2\|f\|_{Lip}.
$$
By \eqref{ineq}, we have
$$
    w_e \leq \left(1 - \frac{c_1 \eta^{\frac{1}{\theta}+1}}{2} \right)\|w_M\|_\infty \quad \mbox{in } \; B_{1/8},
$$
and hence,
$$
    \|w_M\|_{L^\infty(B_{1/16})} \leq \left(1 - \overline{c} \eta^{\frac{1}{\theta}+1} \right)\|w_M\|_\infty. 
$$
Next, we assume that
$$
    c_1 \eta^{\frac{1}{\theta}+1} \|w_M\|_\infty <  2\|f\|_{Lip}.
$$
From this,
$$
	\|w_M\|_{L^\infty(B_{1/16})} \leq \|w_M\|_{\infty} \leq \left(\overline{c} \eta^{\frac{1}{\theta}+1}\right)^{-1}\|f\|_{\Lip}.
$$
The proof is complete. 
\end{proof}

Iterating the previous Lemma in dyadic balls we obtain the following: 

\begin{proposition} \label{small_grad_improv_osc}
Assume $u$ is an effective viscosity solution of \eqref{maineq} and let $\mu,\eta$ be positive constants. 
For some integer $k>0$, we assume that the following holds
$$
    \sup_{e \in \partial B_1}\left|\{x \in B_{2^{-(2i +1)}}\suchthat w_e \geq (1-\eta)\|w_M\|_{L^\infty\left(B_{2^{-2i}}\right)} \}\right| \leq (1-\eta)\left|B_{2^{-(2i + 1)}}\right|
$$
for all $i=1, \cdots, k$. Then, there exists constants $\overline{C}>0$ and $\alpha \in (0,1)$, depending only on $\mu$, $\eta$, $\|f\|_{\Lip}$, $n$, $\lambda$, $\Lambda$, $q$  such that
$$
    \|w_M\|_{L^\infty\left(B_{2^{-2i}}\right)} \leq \max\left(\|Du\|_{L^\infty(B_{1/4})},\overline{C}\right) \, 2^{-2(i-1)\alpha},
$$
for all $i=1, \dots, k+1$.
\end{proposition}

\begin{proof}
We argue by induction. Case $i=1$ is obvious.  We assume that Proposition \ref{small_grad_improv_osc} holds for $i=k$. Let $r = 2^{-2(k-1)}$ and consider
$$
	u_r(x) \coloneqq u(rx)/r.
$$
Easily one notes that $w$ solves \eqref{maineq}, for
$F_r(M) = rF(r^{-1}M)$ and $f_r(x) = rf(rx)$. In addition,
\begin{equation}\nonumber
\begin{array}{c}
	\left|\left\{x \in B_{2^{-2k-1}}\suchthat w_e \geq (1-\eta)\|w_M\|_{L^\infty\left(B_{2^{-2k}} \right)}  \right\}\right| \\[0.4cm] = 2^{-2(k-1)n}\left|\left\{x \in B_{1/8}\suchthat (\partial_e u_r(x) - (\kappa +\mu))_+ \geq (1-\eta)d \right\}\right|,
\end{array}
\end{equation}
where
\begin{equation}\nonumber
	d = \| (|Du_r| - (\kappa + \mu))_+\|_{L^\infty(B_{1/4})} = \|w_M\|_{L^\infty\left(B_{2^{-2k}}\right)}.
\end{equation}
Hence, $u_r$ is under assumptions of Lemma \ref{improv_oscil_step_1}. Therefore,
\begin{equation}\nonumber
\begin{array}{c}
    \displaystyle \|(|D u_r| - (\kappa + \mu))_+\|_{L^\infty\left(B_{1/16}\right)} \leq  \\[0.2cm]
    \displaystyle \max\left\{(1-\overline{c}\,\eta^{\frac{1}{\theta} + 1}) \|(|D u_r| - (\kappa + \mu))_+\|_{L^\infty(B_{1/4})}, \left(\overline{c}\,\eta^{\frac{1}{\theta}+1}\right)^{-1}r\|f\|_{\Lip} \right\}.
\end{array}
\end{equation}
From this,
\begin{equation}
\begin{array}{c}
    \displaystyle \|w_M\|_{L^\infty\left(B_{2^{-2(k+1)}}\right)} \leq \\[0.4cm] 
    \displaystyle \max\left\{(1-\overline{c}\,\eta^{\frac{1}{\theta} + 1}) \|w_M\|_{L^\infty\left(B_{2^{-2k}}\right)}, \left(\overline{c}\,\eta^{\frac{1}{\theta}+1}\right)^{-1}2^{-2(k-1)} \|f\|_{\Lip} \right\}.
\end{array}
\end{equation}

In what follows, we choose 
$$
    \alpha \coloneqq \frac{- \ln\left(1 - \overline{c}\eta^{\frac{1}{\theta} + 1}\right)}{2\ln 2}.
$$
Hence $1-\eta^{\frac{1}{\theta} + 1} \overline{c} = 2^{-2\alpha}$. Utilizing the result for $i=k$, we obtain
$$
    \|w_M\|_{L^\infty\left(B_{2^{-2(k+1)}}\right)} \leq \max\left(\|Du\|_{L^\infty(B_{1/4})},\overline{C}\right) \,2^{-2k\alpha},
$$
which completes the proof. 
\end{proof}

\subsection{Regularity estimates far from the free boundary}\label{subsect large grad} 
In what follows, for a given $\mu \in (0,1)$, we denote
$$
\vartheta_\mu \coloneqq \frac{\mu^{1+q}}{2C},
$$
where $C>0$ is the universal constant given by Lemma \ref{bustelolemma}. 

\begin{lemma}\label{compact_for_large_grad}
Let $u$ be a solution of \eqref{p-regmaineq}, under the conditions
$$
\vartheta \in (0,\vartheta_\mu), \quad \mbox{and} \quad  
|\xi| \geq \kappa + 2\mu.
$$
Given $\varepsilon > 0$, there exists $\varsigma>0$ depending on $\varepsilon$, $\mu$ and $q$ such that, if 
$$
\max\left(\|u\|_\infty, \varsigma^{-1}\|f\|_\infty \right) \leq 1,
$$
then, there exists a $\kappa$-grad viscosity solution to
\begin{equation}\label{harm-close-lemma}
\mathcal{F}(D^2 h) = 0 \quad \mbox{in } \; \{|Dh| > \kappa \} \cap B_{1/2},
\end{equation}
with $\mathcal{F}$ satisfying \eqref{UE}, such that 
$$
\|u - h\|_{L^\infty(B_{1/2})} \leq \varepsilon.
$$
\end{lemma}

\begin{proof}
Let us assume, seeking a contradiction, that the thesis of Lemma fails. That is, for some $\varepsilon_0>0$, there exists a sequence 
$$
    \left (u_k,\vartheta_k,\xi_k,\varsigma_k,f_k,F_k \right )_{ k \in \mathbb{N}}, 
$$ 
where $u_k$ is a normalized solution of \eqref{p-regmaineq}, according to Definition \ref{def-kappa-grad}, with the corresponding parameters given above and  
$$
    \varsigma_k = \text{o}(1),
$$
as $k \to \infty$; however, 
\begin{equation}\label{contradiction_hilton}
    |u_k-h| > \varepsilon_0 \quad \mbox{in } B_{1/2},
\end{equation} 
for all $h$ satisfying \eqref{harm-close-lemma}. From Lemma \ref{bustelolemma}, we have
\begin{equation}\nonumber
    \|Du_k\|_{L^\infty(B_{1/2})} \leq C \mu^{-q}.
\end{equation}
From this, and the fact that $\vartheta_k \leq \vartheta_\mu$ and $|\xi_k|\geq \kappa + 2\mu$, one has
$$
|F_k(D^2u_k)| \leq \|f_k\|_\infty(|\xi_k + \vartheta_kDu_k|-\kappa)_+^{-q} \leq \varsigma_k \mu^{-q}.
$$
Now both $F_k$ and $u_k$ are uniformly bounded and equicontinuous, hence, up to a subsequence, $F_k \to F_\infty$ and $u_k \to u_\infty$ locally uniformly. By stability $u_\infty$ solves 
$$
    F_\infty(D^2 u_\infty) = 0 \quad \mbox{in } \; B_{1/2},
$$
in the $\kappa$-grad viscosity sense. This leads to a contradiction on \eqref{contradiction_hilton} for $k \gg 1$ large enough. 
\end{proof}

The previous Lemma gives proximity to functions that are $\kappa$-grad viscosity solutions, and thus only entitled to local Lipschitz regularity. Next, we show that those functions are actually close to $C^{1,\alpha}$ functions. 

\begin{lemma}\label{bypass_proximity}
Given $\varepsilon>0$, there exists small positive parameters $\kappa$ and 
$\varsigma$, depending on $n$, $\lambda$, $\Lambda$ and $\epsilon$ such that if
$$
    \max\left(\|u\|_\infty,\varsigma^{-1}\|f\|_\infty \right) \leq 1,
$$
and $u$ is a $\kappa$-grad viscosity solution to
$$
    (|Du|-\kappa)_+^qF(D^2 u) = f,
$$
then, there exists $h \in C^{1,\alpha}$ with universal bounds satisfying
$$
    \sup_{B_{1/2}}|u-h|< \epsilon
$$
\end{lemma}

\begin{proof}
Assume, seeking a contradiction, that the Lemma thesis does not hold true. This means there are sequences $u_k,F_k,f_k,\varsigma_k,\kappa_k$ with $\kappa_k$ and $\varsigma_k$ converging to zero, such that $u_k$ is a $\kappa_k$-grad viscosity solution to
$$
    (|Du_k| - \kappa_k)_+^qF_k(D^2u_k) = f_k,
$$
but 
$$
    \sup_{B_{1/2}}|u_k - h| > \epsilon_0,
$$
for some $\epsilon_0>0$ and every $h$ in the set of $C^{1,\alpha}$ functions (with universal bounds to be set a posteriori). 

Since $\varsigma_k\rightarrow 0$, we have $f_k \rightarrow 0$. As $\|u_k\|_\infty \leq 1$, Theorem \ref{LipTHM} yields equicontinuity, and thus, up to a subsequence, we can assume $u_k \rightarrow u_\infty$. Passing a further subsequence, if necessary, $F_k \rightarrow F_\infty$, and, by stability, $u_\infty$ is a $0$-grad viscosity solution to
$$
    |Du_\infty|^qF_\infty(D^2 u_\infty) = 0.
$$
Notice that since the equation is homogeneous, $u_\infty$ solves
$$
    |Du_\infty|^qF_\infty(D^2u_\infty)=0,
$$
and  by \cite[Lemma 6]{IS13}, there holds
$$
    F_\infty(D^2u_\infty) = 0
$$
in the classical viscosity sense. The contradiction follows as in the proof of Lemma \ref{compact_for_large_grad}.
\end{proof}

Next, we use iteration arguments to obtain the following result.

\begin{proposition}\label{reg_for_large_grad}
Let $u$ be a $\kappa$-grad viscosity solution of \eqref{maineq}. There exists constants $\rho_0, \gamma \in (0,1)$ depending on $n$, $\lambda$, $\Lambda$, and small positive constants $\varsigma_0$, $\tau_0$ depending only on $\mu$, $n$, $\lambda$, $\Lambda$ and $q$, such that, if
$$
    \|f\|_{\infty} \leq \varsigma_0,
    \quad
    \mbox{and}
    \quad
    |u(x) - (\xi \cdot x + b)| \leq \tau_0 \quad \mbox{in }\; B_1,
$$
for some $\xi \in \mathbb{R}^n$, such that 
$$
    \kappa+3\mu \leq |\xi|, 
$$
then, for each positive integer $k$, there exists an affine function 
$$
    \ell_k= \xi_k \cdot x + b_k,
$$
such that 
$$
|\xi_{k} - \xi_{k-1}| \leq C\tau_0 \rho_0^{(k-1)\gamma}, \quad |b_{k} - b_{k-1}| \leq C\tau_0 \rho_0^{(k-1)(1+\gamma)}
$$
and
$$
|u-\ell_k| \leq  \tau_0 \rho_0^{(k-1)(1+\gamma)} \quad \mbox{in } \; B_{\rho_0^{k-1}},
$$
for some $C \,\geq 1$ depending on $n$, $\lambda$, $\Lambda$.
\end{proposition}

\begin{proof}
We argue inductively. Case $k=1$ follows from the assumptions, taking  $\xi_0 = \xi_1 = \xi$ and $b_0 = b_1 = b$. Assume that the thesis of the Proposition holds for $k=j$. Define the following function
$$
    u_j(y) \coloneqq \frac{(u-\ell_j)(\rho_0^{j-1} y)}{\tau_0 \rho_0^{(j-1)(1+\gamma)}} \quad \mbox{in } \, B_1.
$$
Note that $u_j$ solves
\begin{equation}\nonumber
(|\xi_j + \tau_0 \rho_0^{(j-1)\gamma}Du_j| -\kappa)_+^qF_j(D^2 u_j) = f_j \quad \mbox{in } \, B_1,
\end{equation}
where 
$$
    F_j(M) = \tau_0^{-1}\rho_0^{(j-1)(1-\gamma)}F(\tau_0\rho_0^{(j-1)(\gamma-1)}M) \,\; \mbox{and} \;\, f_j(x) = \rho_0^{(j-1)(1-\gamma)}\tau_0^{-1}f(\rho_0^{j-1}x).
$$
From the induction thesis, $k=j$, we have $\|u_j\|_\infty \leq 1$. In the sequel, we make the following choice 
\begin{equation}\label{FIRST_COND_ON_MU_0}
\tau_0 \leq \min\left\{\frac{1}{4C}\mu, \, \vartheta_\mu \right\}.
\end{equation}
In addition, take $\varsigma_0$ sufficiently small, such that
$$
    \|f_j\|_\infty \leq \tau_0^{-1}\varsigma_0 = \varsigma,
$$
where $\varsigma$ is given by Lemma \ref{compact_for_large_grad}, for $\varepsilon= \rho_0^{1+\gamma}/2$. Additionally, from \eqref{FIRST_COND_ON_MU_0}
$$
    \sum_{i=1}^{j}|\xi_{i} - \xi_{i-1}| \leq C\tau_0 \sum_{i=1}^{j}\rho_0^{(i-1)\gamma} \leq C\tau_0 \sum_{i=1}^{\infty}\frac{1}{2^i}  \leq \frac{1}{4}\mu, 
$$
provided $\rho_0^\gamma \leq 1/2$. This implies that
$$
    |\xi_j| \geq |\xi| - \sum_{i=1}^{j}|\xi_{i} - \xi_{i-1}| \geq \kappa+2\mu. 
$$
In view of these estimates, we can apply Lemma \ref{compact_for_large_grad} for $u_j$ in combination with Lemma \ref{bypass_proximity}, as to obtain the existence of a  $(\lambda,\Lambda)$-harmonic function $h$, such that
$$
    \|u_j - h\|_{L^\infty(B_{1/2})} \leq \frac{\rho_0^{1+\gamma}}{2}.
$$
Since $h$ is universally bounded, we apply classical regularity estimates, to obtain
$$
    |h(x) - Dh(0) \cdot x - h(0)| \leq C'|x|^{1+\alpha'} \quad \mbox{for } \; x \in B_{1/4},
$$
for constants $C'$ and $\alpha'$ depending upon $n$, $\lambda$ and $\Lambda$. Therefore, selecting 
$$
    \gamma=\alpha'/2 \quad \mbox{and} \quad \rho_0 \leq \min\left\{\left(\frac{1}{2}\right)^{\frac{1}{\gamma}}, \left(\frac{2}{C'}\right)^{\frac{1}{\gamma}}
    \right\},
$$
we obtain
$$
    |h(x) - Dh(0) \cdot x - h(0)| \leq \rho_0^{1+\gamma}/2, \quad  \mbox{for } \; x \in B_{\rho_0}.
$$
By the triangle inequality,
$$
    |u_j(x) - Dh(0) \cdot x - h(0)| \leq \rho_0^{1+\gamma}, \quad \mbox{for } \; x \in B_{\rho_0}.
$$
Finally, we define
$$
    \ell_{j+1}(x) \coloneqq \ell_j(x) - \tau_0\rho_0^{(j-1)(1+\gamma)}\ell(\rho_0^{-(j-1)}x)
$$
where $\ell(x) = Dh(0) \cdot x + h(0)$. Hence,
$$
    |u - \ell_{j+1}| \leq \tau_0\rho_0^{j(1+\gamma)} \quad \mbox{in $B_{\rho_0^{j}}$}, 
$$
which completes the proof.
\end{proof}

\begin{corollary}
Under the assumptions of Proposition \ref{reg_for_large_grad}, there exists a constant $C$ depending only on $n$, $\lambda$ and $\Lambda$, such that 
$$
    |Du(x) - Du(0)| \leq \tau_0 C |x|^\gamma,
$$
for each $x \in B_{1/2}$.
\end{corollary}

\begin{proof}
Recall that 
\begin{equation}\label{COEFDECAY}
    \rho_0^k |\xi_{k+1} - \xi_k|+|b_{k+1} - b_k| \leq 2C\tau_0 \rho_0^{k(1+\gamma)},
\end{equation}
implies that sequences $\xi_k$ and $b_k$ converge. Labeling, 
$$
    \lim_{k \rightarrow \infty}\xi_k = \xi_\infty, \quad \mbox{and} \quad \lim_{k \rightarrow \infty}b_k = b_\infty, 
$$
from \eqref{COEFDECAY}, we obtain
\begin{equation}\nonumber
    |\xi_\infty - \xi_k| \leq \displaystyle \frac{C\tau_0}{1 - \rho_0}\rho_0^{k\gamma}
    \quad \mbox{and} \quad
    |b_\infty - b_k| \leq \displaystyle \frac{C\tau_0}{1 - \rho_0}\rho_0^{k(1+\gamma)}.
\end{equation}
Next, given $r<1$, consider integer $k>0$ such that $\rho_0^{k+1} \leq r \leq \rho_0^k$. Hence, denoting 
$$
    \ell_\infty(x) \coloneqq \xi_\infty \cdot x + b_\infty,
$$
we apply Proposition \ref{reg_for_large_grad}, obtaining so 
\begin{equation}\nonumber
    |u(x) - \ell_\infty(x)| \leq |u(x) - \ell_k(x)| + |\ell_k(x) - \ell_\infty(x)| \leq \displaystyle \tau_0 \frac{1}{\rho_0^{1+\gamma}}\left(1 + \frac{2C}{1-\rho_0} \right)r^{1+\gamma},
\end{equation}
for each $x \in B_{\rho_0^k}$. This implies that
$$
    \sup_{x \in B_r}|u-\ell_\infty|(x) \leq \tau_0 \overline{C} r^{1+\gamma}.
$$
and some constant $\overline{C} = \overline{C}(n,\lambda,\Lambda)$. Therefore,
$$
	|u(x) - \ell_\infty(x)| \leq \tau_0 \overline{C}|x|^{1+\gamma},
$$
for $|x|< 1$. Notice that if we make $x=0$ we get $b_\infty = u(0)$. Furthermore, for $s < 1$ we get
$$
	\left|\frac{u(se_i) - u(0)}{s} - \xi_\infty \cdot \Vec{e}_i \right| \leq \tau_0 \overline{C} s^{\gamma}.
$$
where $e_i$  is a $n$-dimensional canonical vector. Passing to the limit when $s \rightarrow 0$ we obtain that $\xi_\infty \cdot e_i = \partial_{e_i} u(0)$ for every $i = 1,\cdots,n$, and so $\xi_\infty = Du(0)$. Therefore,
$$
	|u(x) - u(0) - Du(0) \cdot x| \leq \tau_0 \overline{C} |x|^{1+\gamma},
$$
for $|x|<1$. In particular,
$$
	|Du(x) - Du(0)| \leq \tau_0 \overline{C} |x|^\gamma,
$$
for $x \in B_{1/2}$.
\end{proof}

\subsection{Proof of Proposition \ref{unif_c1_thm}}\label{subsect proof C1} 
First, for $r\leq 1/2$ and $x_0 \in B_{1/4}$, we define
$$
    u_r(x) \coloneqq \frac{1}{r}u(x_0 + rx),
$$
for $x \in B_1$. Note that we have
\begin{equation}\label{lip_bound}
	\|Du_r\|_{L^\infty(B_{1/2})} \leq \|Du\|_{L^\infty(B_{1/2})} \leq C, 
\end{equation}
where the last estimate is due to Theorem \ref{LipTHM}, for some $C$ depending on dimension, ellipticity and $\|f\|_\infty$. Additionally, we observe that $u_r$ solves
$$
	(|Du_r|-\kappa)_+^q F_r(D^2 u_r) = f_r \quad \mbox{in }\; B_1
$$
for $F_r(M) = rF(r^{-1}M)$ and $f_r(x)=rf(x_0 + rx)$. Next, consider 
$$
    r \coloneqq \frac{1}{1+\|f\|_\infty}\varsigma_0 
$$
for $\varsigma_0$ as in Proposition \ref{reg_for_large_grad}. In the sequel, let
$$
    w_e \coloneqq (Du_r\cdot e - (\kappa +\mu))_+ \quad \mbox{and} \quad w_M \coloneqq (|Du_r| - (\kappa +\mu))_+.
$$
Let $\eta \in (0,1)$ to be chosen later. Define $i_\star \in \mathbb{N}$ to be the smallest parameter $i$ such that
\begin{equation}\nonumber
    \begin{array}{c}
         \displaystyle \sup_{e \in \partial B_1}\left|\left\{x \in B_{2^{-(2i+1)}}\suchthat w_e \geq (1-\eta)\|w_M\|_{L^\infty\left(B_{2^{-2i}} \right)} \right\}\right| \geq \\[0.4cm]
         (1-\eta)|B_{2^{-(2i+1)}}|. 
    \end{array}    
\end{equation}
If $i_\star = \infty$, Proposition \ref{unif_c1_thm} follows directly from Proposition \ref{small_grad_improv_osc}. If, on the other hand, $i_\star < +\infty$, for constants $\overline{C}>0$ and $\alpha \in (0,1)$, there holds
\begin{equation}\label{grad_decay}
    \|w_M\|_{L^\infty\left(B_{2^{-2i}}\right)} \leq \overline{C}2^{-2(i-1)\alpha},
\end{equation}
for all $i=1,2, \cdots, i_\star$. Thus, we can estimate
$$
    \|w_M\|_{L^\infty\left(B_{2^{-2i}}\right)} \leq \|w_M\|_{L^\infty\left(B_{2^{-2i_\star}}\right)} \leq \overline{C} 2^{-2(i_\star-1)\alpha} \leq 4\overline{C}2^{-i\alpha}
$$
for $i=i_\star+1, \cdots, 2i_\star$. From the definition of $i_\star$, there exists at least one direction $e \in \partial B_1$ for which  
$$
    \left|\left\{x \in B_{2^{-2i_\star-1}}\suchthat w_e \geq (1-\eta)\|w_M\|_{L^\infty\left(B_{2^{-2i_\star}}\right)} \right\}\right| \geq (1-\eta)|B_{2^{-2i_\star-1}}|.
$$
Therefore, for
$$
    \overline{v}(x) \coloneqq 2^{2i_\star+1}(u_r(2^{-2i_\star-1}x) - u_r(0)) \quad \mbox{for } \; x \in B_1,
$$
we have that
$$
    |\{x \in B_1\suchthat (\partial_e\overline{v}(x) - (\kappa + \mu))_+ \geq (1-\eta)d_\star \}| \geq (1-\eta)|B_1|,
$$
where $d_\star \coloneqq \|(\partial_e\overline{v}(x) - (\kappa + \mu))_+ \|_{L^\infty(B_1)}$. Additionally, by \eqref{lip_bound} there holds
$$
    |D\overline{v}| \leq \kappa + \mu + d_\star \leq C  \quad \mbox{and} 
\quad \overline{v}(0) = 0.
$$
Considering $\epsilon = \tau_0/C$ and applying \cite[Lemma 4.1]{CF14}, we choose $\eta$ (depending only on the choice of $\epsilon$) to find $(\xi,b)$, such that 
$$
|\xi| = \kappa + \mu + d_\star
\quad 
\mbox{and}
\quad
|\overline{v}(x) - b - \xi \cdot x | \leq \epsilon(\kappa + \mu + d_\star) \leq \tau_0.
$$
We can now apply Proposition \ref{reg_for_large_grad} to obtain
$$
|D\overline{v}(x) - D\overline{v}(0)| \leq C_1|x|^{\gamma}
$$
for $x \in B_{1/2}$. Recall that for $x \in B_{2^{-2i_\star -2}}$ we have
\begin{equation}\nonumber
    |w_M(x) - w_M(0)|  \leq |D\overline{v}(2^{2i_\star+1}x) - D\overline{v}(0)|,
\end{equation}
and thus
\begin{equation}\nonumber
|w_M(x) - w_M(0)|  \leq C_12^{-i\gamma},
\end{equation}
for each $x \in B_{2^{-2i}}$ and $i \geq 2i_\star+1$. 

\smallskip
We are ready to conclude the proof. Setting 
$$
C' = 8 \max\{\overline{C},C_1\} \quad \mbox{and} \quad \overline{\alpha} = \frac{1}{2}\min\{\alpha, \gamma \},
$$
we conclude
$$
\| (|Du(x)| - (\kappa+\mu))_+ - (|Du(x_0)| - (\kappa+\mu))_+ \|_{L^\infty\left(B_{r2^{-2i}}(x_0)\right)} \leq C'2^{-2i\overline{\alpha}},
$$
for every $i \in \mathbb{N}$. Given $x \in B_r(x_0)$, we take integer $j>0$, such
$$
r2^{-2(j+1)} \leq |x-x_0| \leq r2^{-2j}.
$$
This implies that
$$
2^{-2j\overline{\alpha}} \leq \left(\frac{4|x-x_0|}{r} \right)^{\overline{\alpha}}.
$$
We then obtain
\begin{equation}\nonumber
\displaystyle |(|Du(x)| - (\kappa+\mu))_+ - (|Du(x_0)| - (\kappa+\mu))_+ |  \leq \displaystyle C''|x - x_0|^{\overline{\alpha}},\\[0.4cm]
\end{equation}
for $x \in B_r(x_0)$, and constant $C''>0$ depends upon $\mu$, $q$, $n$, $\lambda$, $\Lambda$ and $\|f\|_\infty$. For $x \in B_{1/2}\setminus B_r(x_0)$, we estimate
\begin{equation}\nonumber
\begin{array}{lll}
\displaystyle |(|Du(x)| - (\kappa+\mu))_+ - (|Du(x_0)| - (\kappa+\mu))_+ | & \leq & \displaystyle 2\|Dv\|_{L^\infty(B_1)}\\[0.4cm]
& \leq & \displaystyle  C \,|x - x_0|^{\overline{\alpha}}
\end{array}
\end{equation}
where $C$ is another constant that depends only on $\mu$, $q$, $n$, $\lambda$, $\Lambda$ and $\|f\|_\infty$. Since $x_0$ was taken arbitrary, the proof of Proposition \ref{unif_c1_thm} is finally complete.

\subsection{Concluding the proof of Theorem \ref{C1THM}}\label{sct proof C1} 

Recall that $u$ is an effective viscosity solution of
\begin{equation}\nonumber
    (|Du|-\kappa)_+^q \,F(D^2 u) = f.
\end{equation}
By Theorem \ref{LipTHM}, we have $\|Du\|_\infty \leq C$, for a positive constant 
$$
    C = C(n,\lambda,\Lambda,\kappa, \|u\|_\infty, \|f\|_\infty).
$$ 
By Proposition \ref{unif_c1_thm}, given any $\mu>0$, there exist constants $C_\mu>0$ and $\alpha_\mu \in (0,1)$ depending upon $\mu$ and universal data, such that: 
\begin{equation}\nonumber
	\|(|Du| - (\kappa+\mu))_+\|_{C^{0,\alpha_\mu}(B_{1/4})}  \leq  C_\mu.
\end{equation}
To ease notation define 
$$
    v_\mu(x) = (|Du(x)| - (\kappa+\mu))_+ \quad \text{ and } \quad v(x) = (|Du| - \kappa)_+
$$
By triangle inequality we can estimate:
\begin{equation}\nonumber
\begin{array}{lll}
	\displaystyle |v(x) - v(y)| & \leq & \displaystyle |v_\mu(x) - v(x)| + |v_\mu(y) - v(y)| + |v_\mu(x) - v_\mu(y)| \\
	& & \\
	& \leq & 2\mu + C_\mu|x-y|^{\alpha(\mu)},
\end{array}
\end{equation}
for every $\mu \in (0,1)$. Since such an estimate holds for all $\mu>0$, we obtain
\begin{equation}\nonumber
	|v(x) - v(y)| \leq \sigma(|x-y|),
\end{equation}
where
$$
    \sigma(t) := \min_{\mu \in (0,1)}\{2\mu + C_\mu t^{\alpha(\mu)} \}. 
$$
It is easy to see that $\sigma$, as defined above, is indeed a modulus of continuity and that $Du$ is $\sigma$-continuous within the region $\{ |Du| \geq \kappa \}$.

\section{Applications}\label{sct Applic}
In this final Section we briefly discuss some connections the main Theorems proven in this paper, and the ingredients introduced in their respective proofs, have with other treads of research. 

\subsection{Unconstrained free boundary problems}\label{subsct FB} Initially we revisit the theory of unconstrained free boundary problems, as in the work of Figalli and Shahgholian, \cite{FS14}. 

Let $\Omega$ be an open set of $\mathbb{R}^n$ and $w \in W^{2,p}(B_1)$ be a viscosity solution of 
$$
\begin{cases}
	F(D^2w) = 1 &\quad\text{in $B_1 \cap \Omega$}\\
	|D^2 w| \leq  K & \quad \text{in $B_1 \backslash \Omega$},
\end{cases}
$$
where $F$ is convex and uniformly elliptic. The main result proven in \cite{FS14} is a sharp $C^{1,1}$ regularity of solutions. It is worth comparing such an improved estimate with the results of \cite{Teix13}, where $C^{1, \text{log-Lip}}$ regularity is proven for $F(D^2u) = f \in L^\infty$; see also \cite{CaffH} for related results.

Theorem \ref{LipTHM} can also be viewed as an unconstrained free boundary problem; the first-order counterpart of \cite{FS14}. More precisely, solutions of
$$
\begin{cases}
	F(D^2w) = 1 &\quad\text{in $B_1 \cap \Omega$}\\
	|D w| \leq  K & \quad \text{in $B_1 \backslash \Omega$},
\end{cases}
$$
are $K$-grad viscosity solutions in the sense investigated in this paper. In particular, in the case of linear equations, say $F = \Delta$, Theorem \ref{LipTHM} applied to $w_e$ implies the sharp $C^{1,1}$-regularity of unconstrained free boundary problems at the hessian level. Furthermore, Theorem \ref{C1THM}, applied to $w_e$, yields to the existence of a modulus of continuity $\sigma$ such that $D^2w \in C^{0,\sigma}(\{|D^2 w|\ge K \} \cap B_{1/2})$.

\subsection{PDE models with infinite degree of degeneracy}\label{subsec limiting FB} Next we would like to discuss connections with limiting free boundary problems, obtaining when the degree of degeneracy tends to infinity. More precisely, let us look at the non-variational $q$-Laplacian equation:
\begin{equation}\label{nonv_q_lap}
    |Du|^qF(D^2u) = f.
\end{equation}
This model has received warm attention in the last two decades, see for instance \cite{ART15,BD04,DF21,DF22,IS13,SR20} and references therein. 

An important Corollary of the analysis carried out in Section \ref{sct proof lip} is the following (uniform-in-$q$) sharp regularity estimate:

\begin{corollary}\label{cor-Lip}
Let $q\ge 0$, $f\in L^\infty(B_1)$, and $u_q$ be a normalized viscosity solution of
$$
    |Du_q|^q F(D^2u_q) = f \text{ in }  B_1.
$$
Then, there exists a constant $C$, depending only on dimension, ellipticity, and $\|f\|_{L^\infty(B_1)}$, but independent of $q$, such that
$$
    \|D u_q\|_{L^\infty(B_{1/2})} \le C.
$$
\end{corollary}

Now, let $\{u_q\}_{q>0}$ be a family of normalized viscosity solutions to \eqref{nonv_q_lap}. By Corollary \ref{cor-Lip}, up to subsequence, that $u_q \rightarrow u_\infty$, for some Lipschitz function $u_\infty$. Easily one verifies that $u_\infty$ satisfies:
$$
    F(D^2 u_\infty) = 0, \quad \text{ in } \{|Du_\infty | > 1 \},
$$
that is, $u_\infty$ is a $1$-grad $F$-harmonic function. 

A careful scrutiny of the proofs delivered in this paper yields the following result for the PDE model \eqref{nonv_q_lap}:

\begin{theorem}
Let $q\ge 0$, $f\in \Lip(\overline{B}_1)$, and $u_q$ be a normalized viscosity solution of
$$
    |Du_q|^q F(D^2u_q) = f \text{ in }  B_1.
$$
Then, given $0 < \mu < 1$, there exists constants $0 < \alpha_\mu < 1$ and $C_\mu>0$ depending on data, $\mu$ but independent of $q$ such that
$$
    \|(|Du_q| - (1+\mu))\|_{C^{0,\alpha_\mu}(B_{1/2})} \leq C_\mu.
$$
\end{theorem}

As a further consequence, one obtains that the limiting solution $u_\infty$ has continuous gradient up to the free boundary.

\subsection{Flame propagation with an obstacle}\label{subsct flame propagation} Singularly perturbed PDEs of the flame propagation type have received warm attention since the pioneering work \cite{BCN}, see for instance \cite{ART2, BW, CK, CV, LVW, LW1, LW2, MW, Teix08} and references therein. For free boundary problems driven by operators in non-divergence form, introducing a heavy penalization term, $\beta_\epsilon(u)$, allows for an existence theory, as long as one can obtain strong enough estimates that are uniform with respect to the regularizing parameter $\epsilon$, see for instance \cite{ATeix, DPS, K, RTeix}. 

Typically, $\beta_\epsilon$ is an approximation of the Dirac delta function, $\delta_0$, in $L^1$. One can think of 
$$
    \beta_\epsilon(s) := \frac{1}{\epsilon} \beta \left ( \frac{s}{\epsilon} \right ),
$$  
where $\beta$ is a fixed, smooth function with bounded support. The main goal is to obtain uniform-in-$\epsilon$ regularity estimates for $u$ and its free boundary. 

Here we are interested in a new type of flame propagation models, which carries activation fronts. Mathematically this gives raise to a free boundary problem of the singularly perturbed type for which the jump discontinuity happens along the coincidence set $\Lambda_\epsilon := \{ u_\epsilon = \varphi\}$, for a given obstacle function $\varphi$. 

The starting point of this program is to prove that solutions are uniformly-in-$\epsilon$ Lipschitz continuous, provided the obstacle, $\varphi$, is Lipschitz. This is the result we discuss here as the final application of the methods introduced in this paper.

Hereafter $u_\epsilon$ denotes a viscosity solution of the PDE
\begin{equation}\label{RTeq}
    F(D^2 u_\epsilon) = \beta_\epsilon(u_\epsilon - \varphi),
\end{equation}
with $u_\epsilon \geq \varphi$ and $\varphi \in C^{0,1}$. The main theorem we prove here is the following:

\begin{theorem}
Given $\Omega' \Subset \Omega$, there exists a constant $C'$ such that any bounded family $\{u_\epsilon \}_{\epsilon > 0}$ of solutions of $(\ref{RTeq})$ satisfies
$$
   \|Du_\epsilon\|_{L^\infty(\Omega')} \leq C'\left(n,\lambda,\Lambda,\beta,[\varphi]_{C^{0,1}},\Omega'\right).
$$
\end{theorem}
\begin{proof}
The key feature of the model is its distinct behavior within the regions
$$
    \Omega_1 := \Omega' \cap \{ u_\epsilon - \varphi\leq \epsilon \} \quad \mbox{and} \quad \Omega_2 := \Omega' \cap \{ u_\epsilon - \varphi > \epsilon \}.
$$
By means of a standard covering argument, we can restrict the analysis to the case $\Omega = B_1$ and $\Omega' = B_{1/2}$. 
\medskip
 
\noindent {\bf Case I:} Let $x_0 \in \Omega_1$ be fixed. We will prove the existence of a constant $C_1'>0$ that does not depend on $\epsilon>0$ such that
$$
    |Du_\epsilon (x_0)| \le C_1'\left (n,\lambda,\Lambda,\beta,[\varphi]_{C^{0,1}} \right ).
$$
For that, define the auxiliary function, $v \colon B_2 \to \mathbb{R}$, as:
$$
    v(z) \coloneqq \epsilon^{-1}[u_\epsilon(x_0 + \epsilon z) - u_\epsilon(x_0)].
$$
Direct calculations show that $v$ solves
$$
    F_\epsilon(D^2v) = \beta(v-\tilde{\varphi}),
$$
in $B_2$, where $\tilde{\varphi}(z) \coloneqq \epsilon^{-1}(\varphi(x_0 + \epsilon z) - u_\epsilon(x_0))$ and $ F_\epsilon(M) \coloneqq \epsilon F(\epsilon^{-1}M)$. Note that the equation for $v$ is uniformly elliptic and therefore Lipschitz estimates are available. In particular we can estimate
$$
    |Du_\epsilon(x_0)| = |Dv(0)|\le C \|v\|_{L^\infty(B_{3/2})},
$$
for a constant $C>0$ depending only on $n$, $\lambda$, $\Lambda$ and $\|\beta\|_\infty$. 

Now we turn to get uniform (in the parameter $\epsilon$) estimates for $\|v\|_{L^\infty(B_{3/2})}$. Recall that since $u_\epsilon \geq \varphi$ and $u_\epsilon(x_0) - \varphi(x_0) \leq \epsilon$, we get for $z \in B_2$,
$$
    \begin{array}{lll}
       v(z)  & = & \epsilon^{-1}[u_\epsilon(x_0 + \epsilon z) - u_\epsilon(x_0)]\\
       & & \\
         &\geq & \epsilon^{-1}[\varphi(x_0 + \epsilon z) - \varphi(x_0) - \epsilon]\\
         & & \\
         & \geq & - \epsilon^{-1}|\varphi(x_0 + \epsilon z) - \varphi(x_0)| - 1\\
         & & \\
         & \geq & -[\varphi]_{\mathcal{C}^{0,1}}|z| - 1\\
         & & \\
         & \geq & -2\left( [\varphi]_{\mathcal{C}^{0,1}} + 1 \right) = -\kappa
    \end{array}
$$
Harnack inequality applied to the non-negative function $w:= v +\kappa \geq 0$ yields
$$
\begin{array}{lll}
    \displaystyle \sup_{B_{3/2}}w  & \leq & \displaystyle C\left(w(0) + \|\beta\|_{L^\infty(\mathbb{R})}\right)\\[0.5cm]
     & \leq & \displaystyle C\left(\kappa + \|\beta\|_{L^\infty(\mathbb{R})}\right).
\end{array}
$$
Combining all such estimates we finally end up with
$$
    \|Du_\epsilon\|_{L^\infty(\Omega_1)} \leq K.
$$
for $K$ depending on $n$, $\lambda$, $\Lambda$, $\|\beta\|_\infty$ and $[\varphi]_{C^{0,1}}$.

\medskip
\noindent {\bf Case II:} The estimate for $x_0 \in \Omega_2$.

We simply note that, in view of the estimate obtained in Case I, $u_\epsilon$ satisfies 
$$
    F(D^2 u_\epsilon) = 0 \quad \text{ in } \{ |Du_\epsilon| > K\}.
$$
Theorem \ref{LipTHM} then gives the desired local Lipschitz estimate, independently of the parameter $\epsilon>0$.
\end{proof}

\textbf{Acknowledgments.} DJA is partially supported by CNPq 311138/2019-5 and grant 2019/0014 Paraíba State Research Foundation (FAPESQ). This study was financed in part by the Coordenação de Aperfeiçoamento de Pessoal de Nível Superior - Brasil (CAPES) - Finance Code 001.

\end{document}